\documentclass{amsart}

\usepackage[hyperindex]{hyperref}

\newcommand{\balg}{\begin{algorithm}}
\newcommand{\ealg}{\end{algorithm}}
\newcommand{\br}{\begin{remark}}
\newcommand{\er}{\end{remark}}
\newcommand{\bex}{\begin{example}}
\newcommand{\eex}{\end{example}}

\newtheorem{theorem}{Theorem}[section]

\theoremstyle{definition}

\newtheorem{example}[theorem]{Example}

\newtheorem{algorithm}[theorem]{Algorithm}

\newtheorem{remark}[theorem]{Remark}

\numberwithin{equation}{section}

\begin{document}

\title{The CP-matrix Approximation Problem}

\author{Jinyan Fan}
\address{Department of Mathematics, and MOE-LSC, Shanghai Jiao Tong University,
Shanghai 200240, P.R. China}
\email{jyfan@sjtu.edu.cn}
\thanks{The first author is partially supported by NSFC 11171217.}
%

\author{Anwa Zhou}
\address{Department of Mathematics, Shanghai Jiao Tong University,
Shanghai 200240, P.R. China}
\email{congcongyan@sjtu.edu.cn}
\thanks{}

\subjclass[2000]{Primary: 90C20, 90C22, 90C26}

\date{}

\dedicatory{}

\begin{abstract}
A symmetric matrix $A$ is completely positive (CP) if there exists an entrywise
nonnegative matrix $V$ such that $A = V V ^T$.
In this paper, we study the CP-matrix approximation problem
 of projecting a matrix onto the intersection of a set of linear constraints
 and the cone of CP matrices.
 We formulate the problem as the linear optimization
 with the norm cone and the cone of moments.
 A semidefinite algorithm is presented for the problem.
A CP-decomposition of the projection matrix can also be obtained if the problem is feasible.

\end{abstract}

\keywords{completely positive matrices, CP projection, CP-matrix approximation,
linear optimization with moments, semidefinite algorithm}

\maketitle

\section{Introduction}
A real $n\times n$ symmetric matrix $A$ is completely
positive (CP) if there exist nonnegative vectors $v_1,\cdots,v_r
\in \mathbb{R}^n_+$ such that
\begin{equation}\label{CPd}
A = v_1 v_1^T +\cdots +v_r v_r^T,
\end{equation}
where $r$ is called the length of the decomposition \eqref{CPd}.
The smallest $r$ in the above is called the CP-rank of $A$.
If $A$ is CP, we call (\ref{CPd}) a CP-decomposition of $A$.
So, $A$ is CP if and only if $A=VV^T$ for an entrywise nonnegative $V$.
Clearly, a CP-matrix is double nonnegative, i.e., it is not only positive semidefinite but also nonnegative entrywise.

Let $\mathcal{S}_n$ be the set of real $n\times n$ symmetric matrices.
For a cone $\mathcal{C}\subseteq \mathcal{S}_n$, the dual cone of $\mathcal{C}$ is defined as
$$
\mathcal{C}^*:= \{B \in \mathcal{S}_n :  A\bullet B \geq 0 \; \text{for all} \; A \in \mathcal{C}\},
$$
where $A\bullet B:=\text{trace}(A^T B)$ is the standard inner product on $\mathbb{R}^{n\times n}$. Denote
\begin{align*}
& \mathcal{CP}_n = \{A \in \mathcal{S}_n : A= VV^T \ \text{with}\ V\geq 0 \}, \text{the completely positive cone},\\
& \mathcal{COP}_n = \{ B \in \mathcal{S}_n: x^T B x\geq 0\ \text{for all}\ x \geq0\}, \text{the copositive cone}.
\end{align*}
Both $\mathcal {CP}_n$ and $\mathcal {COP}_n$
are proper cones (i.e. closed, pointed, convex and
full-dimensional).  Moreover, they are dual to each other \cite{Hall1967}.
A variety of NP-hard problems can be formulated as optimization problems over
the completely positive cone or the copositive cone.
Interested readers are referred to \cite{BermanN,Burer,Burer12,Chen12,Bomze1,deKlerk,Dur} for the work in the field.

The important applications of the CP cone motivate people to study whether a matrix is CP or not.
However, checking the membership in $\mathcal{CP}_n$ has been shown NP-hard,
while checking the membership in $\mathcal{COP}_n$ co-NP-hard \cite{Dickinson11,Murty}.
It is generally difficult to treat $\mathcal {CP}_n$ (or $\mathcal{COP}_n$) directly.
A standard approach is to approximate it by simpler and more tractable cones
\cite{deKlerk, DongA,Bomze02,Pena,Lasserre2014}.
By Nie's approach proposed in \cite{Nie, Nie1},
Zhou and Fan \cite{ZhouFan14a} presented a semidefinite algorithm for the CP-matrix completion problem,
which includes the CP checking as a special case;
a CP-decomposition for a general CP-matrix can also be found by the algorithm.
The approach is also applied to check interiors of the completely positive cone \cite{ZhouFan14b}.


In \cite{SponselD}, Sponseldur and D\"{u}r considered the problem of projecting a matrix onto the cones of copositive and completely positive matrices.
   Unlike projecting onto  the cones of nonnegative matrices and positive semidefinite matrices, projecting onto either  $\mathcal{CP}_n$ or  $\mathcal{COP}_n$ is a nontrivial task in view of the NP-complexity results  \cite{Dickinson11, Murty}.
 Sponseldur and D\"{u}r used polyhedral approximations of $\mathcal {COP}_n$
 to compute the projection of a matrix onto $\mathcal {COP}_n$ and the projection onto $\mathcal {CP}_n$ by a dual approach.

In this paper, we consider the  general CP-matrix approximation problem stated as:
\begin{align}\label{CPLS}
 \begin{array}{rll}
  \displaystyle \min_X  & \|X-C\|_p\\
  \mbox{s.t.} &  A_i\bullet X = b_i, \ i=1,\ldots,m_e, \\
              &   A_i\bullet X \geq b_i, \ i=m_e+1,\ldots,m, \\
              & X \in \mathcal{CP}_n,
 \end{array}
\end{align}
where $C, A_i\in\mathcal{S}_n, b_i \in \mathbb{R} (i=1,\ldots,m)$,
and $\|\cdot\|_p$ is the $p$-norm ($p=1,2,\infty$ or $F$).
The problem is projecting a symmetric matrix onto the intersection of a set of linear constraints and the complete positive cone.

Specially, if $C=0$, then  \eqref{CPLS} becomes the feasibility problem of finding a matrix in the intersection of a set of linear constraints and the CP cone, which has the minimum $p$-norm.

If there are no linear constraints, \eqref{CPLS} is reduced to
the CP projection problem
\begin{align}\label{cpp}
 \begin{array}{rl}
 \displaystyle \min_X  & \|X-C\|_p\\
  \mbox{s.t.} &   X \in \mathcal{CP}_n.
 \end{array}
\end{align}
Hence, the CP projection problem is a special case of \eqref{CPLS}.
Clearly, \eqref{cpp} is always feasible and has a  solution.
If the minimum is zero, then the projection matrix of $C$ onto $\mathcal{CP}_n$ is itself, which implies that $C$ is CP.
If the minimum is nonzero, then $C$ is not CP.
So, solving the CP projection problem \eqref{cpp} provides a way to check whether $C$ is CP.

In this paper, we formulate \eqref{CPLS} as a linear optimization problem with the cone of moments and the $p$-norm cone, then propose a semidefinite algorithm for it.
If \eqref{CPLS} is infeasible, we can get a certificate for it.
If \eqref{CPLS} is feasible, we can get a projection matrix of $C$ onto the set of linear constraints and the CP cone. Moreover, a CP-decomposition of the projection matrix can also be obtained.

The paper is organized as follows.
In section 2, we review the norm cone and its dual cone, and characterize the CP matrix as a moment sequence. In section 3, we show how to formulate \eqref{CPLS} as a linear optimization problem with the norm cone and the cone of moments; its dual problem is also given.
We present a smidefinite algorithm for \eqref{CPLS} and study its convergence properties in section 4. Some computational results are given in section 5.
Finally, we conclude the paper  in section 6.

\section{Preliminaries}
In this section, we first give the dual norm of the $p$-norm on $\mathcal{S}_n$ for $p=1,2,\infty$ and $F$ respectively; the $p$-norm cone and its dual cone are also given.
 Then we characterize CP matrices as moments, and review some basics about moments and localizing matrices, as well as the semidefinite relaxations of the CP cone (cf. \cite{Lasserre01,Lasserre08, Laurent, Lasserre, Nie}).

\subsection{$p$-norm cone and its dual}
For $A\in \mathbb{R}^{n\times n}$, the $p$-norms of $A$ ($p=1,2,\infty, F$) are defined by:
\begin{align*}
&\|A\|_1=\max\limits_{j}\sum_{i=1}^n |A_{ij}|, \quad \mbox{the maximum absolute column sum norm or}\  \mbox{1-norm},  \\
&\|A\|_2= (\lambda_{\max} (A^T A))^{1/2}, \quad \mbox{the spectral norm or}\  \mbox{2-norm},  \\
&\|A\|_\infty= \max\limits_{i}\sum_{j=1}^n |A_{ij}|, \quad \mbox{the maximum absolute row sum norm or}\ \mbox{$\infty$-norm},  \\
&\|A\|_F=(\operatorname{trace}(A^T A))^{1/2},\quad \mbox{the Frobenius norm or}\ F\mbox{-norm}.
\end{align*}
Note that, when $A\in \mathcal{S}_n$, the $1$-norm is the same as $\infty$-norm.

Let $\| \cdot \|$ be a norm on $\mathcal{S}_n$.
 The associated dual norm, denoted by $\| \cdot \|_{*}$, is defined by
\begin{align}\label{norm}
\|A\|_{*}= \sup\{A\bullet X : \|X\| \leq 1\},
\end{align}
(cf. \cite[Section A.1.6]{Boyd}).
It can be proved that the dual norm of the $p$-norm for  $p=1,2,\infty$ and $F$ are:
\begin{align*}
\|A\|_{1*}&= \sum_{j=1}^n \max\limits_{i}|A_{ij}|,\\
\|A\|_{2*}&=\operatorname{trace}((A^T A)^{1/2}),\\
\|A\|_{\infty *}&= \sum_{i=1}^n \max\limits_{j}|A_{ij}|,\\
\|A\|_{F*}&=\|A\|_F.
\end{align*}

For $\| \cdot \|$ on $\mathcal{S}_n$, the norm cone is defined by
$$
\mathcal{K} = \{(X, s) \in \mathcal{S}_n \times \mathbb{R}_+ : \| X \|\leq s\},
$$
where $\mathcal{S}_n \times \mathbb{R}_+$ is the Cartesian product of $\mathcal{S}_n$ and $\mathbb{R}_+$.
The dual cone of $\mathcal{K}$ is defined by
$$
\mathcal{K}^*= \{(Y, t) \in \mathcal{S}_n \times \mathbb{R}_+ :  X\bullet Y+st \geq 0 \; \text{for all} \; (X, s) \in \mathcal{S}_n \times \mathbb{R}_+  \}.
$$
For the $p$-norm cone ($p=1,2,\infty$ and $F$),
$$
\mathcal{K}_p = \{(X, s) \in \mathcal{S}_n \times \mathbb{R}_+ : \| X \|_p\leq s\},
$$
we can prove that the dual cone of $\mathcal{K}_p$ is
\[
\mathcal{K}_p^* = \{(Y, t) \in \mathcal{S}_n \times \mathbb{R}_+ : \| Y \|_{p*} \leq t\}.
\]


\subsection{Characterization as moments}
A symmetric matrix $A\in \mathcal{S}_n$ can be identified by a vector consisting of its upper triangular entries, i.e.
$$
\operatorname{vech}(A)=(A_{11},A_{12},\ldots,A_{1n},A_{22},\ldots,A_{2n},A_{33},\ldots,A_{nn})^T.
$$
Let $\mathbb{N}$ be the set of nonnegative integers.
For $\alpha = (\alpha_1,\cdots, \alpha_{n}) \in \mathbb{N}^{n}$,
denote $|\alpha| := \alpha_1+\cdots+\alpha_{n}$.
Let
\begin{equation}\label{AE}  E  := \{\alpha
\in \mathbb{N}^{n}:\, |\alpha|=2 \}.
\end{equation}
Then, $A$ can also be identified as
$$
a=(a_{\alpha})_{\alpha\in  E } \in \mathbb{R}^ E ,\quad
a_{\alpha}=A_{ij} \; \text{if} \; \alpha=e_i+e_j, i\leq j,
$$
where $e_i$ is the $i$-th unit vector in $\mathbb{R}^n$ and $\mathbb{R}^{ E }$ denotes the space of real vectors indexed by  $\alpha \in  E $.
We call $a$ an $ E $-truncated moment sequence ($ E  $-tms).

Let
\begin{equation} \label{KE}
\Delta=\{x\in \mathbb{R}^{n} :\,
x_1^2+\cdots +x_{n}^2 -1= 0, x_1 \geq 0, \cdots, x_{n} \geq 0\}
\end{equation}
be the nonnegative part of the unit sphere.
 Every nonnegative vector is a multiple
of a vector in  $\Delta$.
So, by (\ref{CPd}),
$A\in \mathcal{CP}_{n}$ if and only if
there exist  $\rho_1, \cdots, \rho_r >0$ and
$u_1,\cdots, u_r \in \Delta$ such that
 \begin{equation}\label{ECPe}
A=\rho_1 u_1 u_1^T +\cdots +\rho_l u_r u_r^T.
 \end{equation}

 The $ E $-truncated $\Delta$-moment problem
($ E $-T$\Delta$MP) studies whether or not a given {$ E $-tms $a$ admits a
$\Delta$-measure} $\mu$, i.e., a nonnegative Borel measure $\mu$ supported
in $\Delta$ such that
$$
a_{\alpha}= \int_\Delta x^{\alpha} d \mu, \quad
\forall\, \alpha \in   E ,
$$
where $x^{\alpha} := x^{\alpha_1}_1 \cdots x^{\alpha_{n}}_{n}$.
A measure $\mu$ satisfying the above is called a
{$\Delta$-representing measure} for $a$.
A measure is called {finitely
atomic} if its support is a finite set, and is called $r$-atomic if its
support consists of at most $r$ distinct points.

Hence, by \eqref{ECPe}, a symmetric matrix $A$, with the identifying vector
$a\in \mathbb{R}^{ E }$,  is completely positive if and only if $a$ admits
an $r$-atomic $\Delta$-measure,  i.e.,
\begin{equation}
\label{ECPee} a=\rho_1 [u_1]_{ E } +\cdots +\rho_r [u_r]_{ E },
\end{equation}
where each $\rho_i>0$, $u_i \in \Delta$  and
\[
[u_i]_{ E } :=(u_i^{\alpha})_{\alpha \in { E }}, \quad i=1,\ldots,r.
\]

Denote
\begin{align}\label{CPR}
\mathcal{R}=\{a\in \mathbb{R}^ E :  a\ \mbox{admits a}\ \Delta\mbox{-measure} \}.
\end{align}
Then, $\mathcal{R}$ is the CP cone (cf. \cite{Nie1}).
Hence,
\begin{align}\label{CPE1}
A\in \mathcal{CP}_{n} \quad \mbox{if and only if}\quad  a\in  \mathcal{R}.
\end{align}

\subsection{Localizing matrices and flatness}
Let $ E $ and $\Delta$ be given in \eqref{AE} and \eqref{KE}, respectively.
Denote
\[
\mathbb{R}[x]_{ E }:= \mbox{span}\{x^{\alpha}: \alpha\in  E \}.
\]
We say $\mathbb{R}[x]_{ E }$ is $\Delta$-full if there exists a polynomial $p \in \mathbb{R}[x]_{ E }$ such that $p|_\Delta>0$ (i.e. $p(u)>0$ for all $u\in \Delta$).
It is shown in \cite{Nie 1} that the dual cone of $\mathcal{R}$ is
\begin{align}\label{COP}
\mathcal{P}=\{p\in \mathbb{R}[x]_{ E }: p(x)\geq 0, \forall x\in \Delta\}.
\end{align}

An $ E $-tms $a \in \mathbb{R}^{ E }$
defines a Riesz functional $F_{a}$ acting on
$\mathbb{R}[x]_{ E }$ as
 \begin{equation}\label{La} {F}_{a}
(\sum_{\alpha\in  E }p_{\alpha} x^{\alpha}):=
\sum_{\alpha\in  E }p_{\alpha} a_{\alpha}.
 \end{equation}
 For convenience, we also denote the inner product $\langle p, a \rangle:= {F}_{a}(p)$.

Let
$$
\mathbb{N}_d^{n} := \{ \alpha \in \mathbb{N}^{n}: \, |\alpha| \leq d\}
\quad\mbox{and}\quad
\mathbb{R}[x]_{d}:=\mbox{span}\{x^{\alpha}: \alpha\in \mathbb{N}^{n}_{d}\}.
$$
For $s \in \mathbb{R}^{\mathbb{N}^{n}_{2k}}$ and $q
\in \mathbb{R}[x]_{2k}$, the $k$-th localizing matrix  of
$q$ generated by $s$ is the symmetric
matrix $L^{(k)}_q (s)$ satisfying
 \begin{equation}\label{Lzqp2}
{F}_s (q p^2)= \operatorname{vec}(p)^T (L^{(k)}_q (s))\operatorname{vec}(p), \quad \forall p\in
\mathbb{R}[x]_{k-\lceil deg(q)/2 \rceil}.
 \end{equation}
In the above, $\operatorname{vec}(p)$ denotes the coefficient vector of polynomial $p$ in the
graded lexicographical ordering, and $\lceil t\rceil$
denotes the smallest integer that is not smaller than $t$.
In particular, when $q=1$, $L^{(k)}_1 (s)$ is called a  $k$-th
order moment matrix  and denoted as $M_k(s)$. We refer to \cite{Nie,Helton,Nie4} for
more details about localizing and moment matrices.

Denote the polynomials:
\begin{align*}
h(x):=  x_1^2+\cdots +x_{n}^2-1,g_0(x):=1,
 g_1(x): =  x_1, \cdots,  g_{n}(x): = x_{n}.
\end{align*}
 Note that $\Delta$ given in \eqref{KE} is nonempty compact.
 It can also be described equivalently as
\begin{equation}\label{K2}
\Delta=\{x\in \mathbb{R}^{n}: \ h(x)= 0, g(x) \geq 0\},
\end{equation}
where $g(x)=(g_0(x), g_1(x),\cdots,g_{n}(x))$.
As shown in \cite{Nie}, a necessary condition for $s \in \mathbb{R}^{\mathbb{N}^{n}_{2k}}$
to admit a $\Delta$-measure is
\begin{equation}
\label{SDPC}
  L^{(k)}_{h} (s) = 0 \quad \mbox{and}\quad L^{(k)}_{g_j} (s) \succeq
0, \quad j=0,1,\cdots,n,
\end{equation}
where $L^{(k)}_{g_j} (s) \succeq 0$ means $L^{(k)}_{g_j} (s)$ is symmetric positive semidefinite.
If, in addition to (\ref{SDPC}), $s$ satisfies the {rank condition}
\begin{equation}
\label{RC}
\text{rank} M_{k-1}(s) =\text{rank} M_{k} (s),
\end{equation}
then $s$ admits a unique
$\Delta$-measure, which is $\text{rank} M_k(s)$-atomic
(cf. Curto and Fialkow \cite{CurtoF}).
We say   $s$ is {flat} if both (\ref{SDPC}) and (\ref{RC}) are satisfied.

Given two tms' $\bar a \in \mathbb{R}^{\mathbb{N}^{n}_{d}}$ and $\bar{\bar{a}} \in
\mathbb{R}^{\mathbb{N}^{n}_{e}}$, we say $\bar{\bar{a}}$ is an  extension  of $\bar a$, if $d\leq e$ and
$\bar a_{\alpha} = \bar{\bar{a}}_{\alpha}$ for all $\alpha \in \mathbb{N}^{n}_{d}$.
If $\bar{\bar{a}}$ is flat and extends $\bar a$, we say $\bar{\bar{a}}$ is a {flat
extension} of $\bar a$.
 We denote
by $\bar{\bar{a}}|_{ E }$ the subvector of $\bar{\bar{a}}$, whose entries are indexed by
$\alpha \in  E $.
Note that an $ E $-tms $a\in \mathbb{R}^{ E }$ admits a $\Delta$-measure
if and only if it is extendable to a flat tms $\tilde a \in \mathbb{R}^{\mathbb{N}^{n}_{2k}}$
for some $k$ (cf. \cite{Nie}). By (\ref{CPE1}), we have that
\begin{align}\label{CPE2}
A\in \mathcal{CP}_n \quad \mbox{if and only if}\quad a \ \mbox{has a flat extension}.
\end{align}

\subsection{Semidefinite relaxations}
We call a subset $I \subseteq \mathbb{R}[x]$ an ideal if $I + I \subseteq I$ and $I \cdot \mathbb{R}[x]\subseteq I$.
For a tuple $\chi = (\chi_1, \ldots, \chi_m)$ of
polynomials in $\mathbb{R}[x]$, denote  by $I(\chi)$ the ideal generated by $\chi_1,\ldots , \chi_m$.
 The smallest ideal containing all $\chi_i$ is
the set $\chi_1 \mathbb{R}[x] + \cdots + \chi_m \mathbb{R}[x]$.
A polynomial $f \in \mathbb{R}[x]$ is called a sum of squares
(SOS) if there exist $f_1, \ldots , f_k \in \mathbb{R}[x]$ such that $f = f^2_1 + \cdots + f^2_k$.

Let $h$ and $g$ be as in \eqref{K2}. Denote
\begin{equation}\label{I2k}
  I_{2k}(h) = \left\{h(x)\phi(x): \operatorname{deg}(h\phi)\leq 2k \right\},
\end{equation}
and
\begin{equation}\label{Qk}
  Q_k (g)= \left\{ \sum_{j=0}^{n} g_j \varphi_j : \text{each} \; \operatorname{deg}(g_j \varphi_j)\leq 2k \; \text{and} \; \varphi_j \; \text{is SOS}\right\}.
\end{equation}
Then, $I(h)=\bigcup_{k\in \mathbb{N}} I_{2k}(h)$ is the ideal generated by $h$,
and $Q(g)= \bigcup_{k\in \mathbb{N}} Q_k (g)$ is the quadratic module generated by $g$ (cf. \cite{Nie1}).
We say $I(h)+ Q(g)$ is archimedean if there exists $R> 0$ such that $R- \|x\|^2 \in I(h)+Q(g)$.
Clearly, if $f\in I(h) + Q(g)$, then $f|_\Delta \geq 0$.
Conversely, if $f|_\Delta > 0$ and $I(h) + Q(g)$ is archimedean, then $f\in I(h) + Q(g)$.
 This is due to Putinar¡¯s Positivstellensatz (cf. \cite{Putinar}).

For each $k \in \mathbb{N}$, denote
\begin{equation}\label{DDk}
\Psi_k = \left\{ p\in \mathbb{R}[x]_{ E }
  : p \in I_{2k}(h)+Q_k (g) \right\}.
\end{equation}
Note that $ E $ is finite, $\mathbb{R}[x]_{ E }$ is $\Delta$-full because $p=\sum_{i=1}^{n}x_i^2|_\Delta>0$,
and $I(h)+ Q(g)$ is archimedean because $1- \|x\|^2 = -h(x) \in I(h)+Q(g)$.
By \cite[Propositions 3.5]{Nie1}, we have
\begin{equation}\label{h0}
\Psi_{1}  \subseteq \cdots \subseteq \Psi_{k}  \subseteq
\Psi_{k+1}  \subseteq \cdots \subseteq \mathcal{P}.
\end{equation}
Moreover,
\begin{equation}
\operatorname{int}(P)  \subseteq  \bigcup_{k} \Psi_k  \subseteq  P.
\end{equation}

Correspondingly, for each $k \in \mathbb{N}$, denote
\begin{equation}\label{h11}
  \Gamma_k = \left\{s \in \mathbb{R}^{\mathbb{N}^{n}_{2k}}: L^{(k)}_{h} (s) = 0, L^{(k)}_{g_j} (s) \succeq 0, j=0,1,\cdots,n
  \right\},
\end{equation}
and
\begin{equation}\label{h12}
  \Upsilon_k = \left\{ s|_{ E }
  : s \in \Gamma_k
  \right\},
\end{equation}
(If $k < deg( E )/2$, $\Upsilon_k$ is defined to be $\mathbb{R}^{ E }$, by default).
Since  $ E $ is finite, $\mathbb{R}[x]_{ E }$ is $\Delta$-full and $I(h)+ Q(g)$ is archimedean, by \cite[Proposition 3.3]{Nie1}, we have
\begin{eqnarray}\label{h13}
\Upsilon_{1}  \supseteq \cdots \supseteq \Upsilon_{k}  \supseteq
\Upsilon_{k+1}  \supseteq \cdots \supseteq\mathcal{R},
\end{eqnarray}
and
\begin{equation}\label{h14}
\bigcap_{k=1}^{\infty} \Upsilon_{k}  = \mathcal{R}.
\end{equation}
Moreover, $\Psi_k$ and $\Upsilon_{k}$ are dual to each other (cf. \cite{Lasserre, Laurent, Nie1}).

As shown above,  the hierarchy of $\Upsilon_k$ provides the outer approximations of $\mathcal{R}$
 and converges monotonically and asymptotically to $\mathcal{R}$.
So, $\Upsilon_k$ can approximate the completely positive cone $\mathcal{R}$ arbitrarily well.

\section{Linear optimization with the CP cone and norm cone}\label{formul}
In this section, we formulate the CP-matrix approximation problem \eqref{CPLS} as a linear optimization problem with the cone of moments and the $p$-norm cone. The duality is also discussed.

Introducing a variable $\gamma\in \mathbb R_+$,
we transform \eqref{CPLS} to the following problem:
\begin{align}\label{CPLS1}
\begin{array}{rl}
 \displaystyle  \min\limits_{X, \gamma}  & \gamma \\
  \mbox{s.t.} &  \| X-C \|_p\leq \gamma,\\
 &  A_i\bullet X = b_i, \  i=1,\ldots,m_e,\\
 &  A_i\bullet X \geq b_i, \ i=m_e+1,\ldots,m,  \\
  & X \in \mathcal{CP}_n.
 \end{array}
\end{align}
Let $Y=X-C$. Then \eqref{CPLS1} can be equivalently written as
 \begin{align}\label{CPLS1b}
\begin{array}{rl}
 \displaystyle  \min\limits_{X, Y, \gamma}  & \gamma \\
  \mbox{s.t.} &  A_i\bullet X = b_i,\ i=1,\ldots,m_e, \\
 &  A_i\bullet X \geq b_i,\ i=m_e+1,\ldots,m,  \\
   & X-Y=C,\\
   & X \in \mathcal{CP}_n,\\
   & (Y,\gamma)\in \mathcal{K}_p.
 \end{array}
\end{align}
The Lagrange function of \eqref{CPLS1b} is:
\begin{align}
 L(X, Y, \gamma,\lambda, P, S, Z, \xi)
 =& \gamma-\sum_{i=1}^{m_e}\lambda_i(A_i\bullet X-b_i)
 -\sum_{i=m_e+1}^m\lambda_i(A_i\bullet X-b_i)\\
  & -  (X- Y- C)\bullet P
   -X\bullet S   -Y\bullet Z -\gamma\xi.\nonumber
 \end{align}
Denote by $\mathcal{F}\eqref{CPLS1b}$ the feasible set of \eqref{CPLS1b}.
 Then, the Lagrange dual problem of \eqref{CPLS1b} is
 \begin{align}\label{CPLSD}
 \begin{array}{cl}
 \displaystyle
\max_{\lambda, P, S, Z, \xi} &\inf\limits_{(X, Y, \gamma)\in \mathcal{F}\eqref{CPLS1b}} L(X, Y, \gamma, \lambda, P, S, Z, \xi)\\
 \mbox{s.t.} & \lambda_{i}\geq 0,\ i=m_e+1,\ldots m,\\
 & P\in \mathcal{S}_n,\\
 & S\in \mathcal{COP}_n,\\
 & (Z,\xi)\in \mathcal{K}_p^*.
 \end{array}
\end{align}
Let $b=(b_1,\ldots,b_m)^T$.
Then \eqref{CPLSD} can be simplified as:
 \begin{align}\label{CPLSDb}
 \begin{array}{cl}
 \displaystyle
\max_{\lambda,  S, Z} & b^T \lambda +C\bullet Z\\
 \mbox{s.t.} &\sum\limits_{i=1}^m\lambda_i A_i+ S  +Z = 0,\\
 & \lambda_i\geq 0, \ i=m_e+1,\ldots,m,\\
 & (S, (Z,1))\in \mathcal{COP}_n\times \mathcal{K}_p^*.
  \end{array}
\end{align}


 Denote
\begin{align*}
x=& \operatorname{vech}(X)\in \mathbb{R}^{\bar n}\quad \mbox{with}\quad \bar n=n(n+1)/2,\\
 a_i=&\operatorname{vech}(2E_n-I_n)\circ \operatorname{vech}(A_i)\in \mathbb{R}^{\bar n}, \  i=1,\ldots, m,
\end{align*}
where ``$\circ$" denotes the Hadamard product, $E_n$ is the all-ones matrix and $I_n$ the identity matrix of order $n$ respectively,
Then, \eqref{CPLS1b} can be formulated as the following linear optimization problem:
\begin{align*}
(P): \qquad\quad
\begin{array}{rl}
 \vartheta_P =  \min\limits_{x, Y,\gamma}  &  \gamma\\
  \mbox{s.t.}   &   a_i^T x   = b_i,\ i=1,\ldots,m_e, \\
 &   a_i^T x  \geq b_i,\ i=m_e+1,\ldots,m, \\
 & x-\operatorname{vech}(Y)=\operatorname{vech}(C),\\
  & (x,(Y,\gamma)) \in  \mathcal{R} \times \mathcal{K}_p,
 \end{array}
\end{align*}
where $\mathcal{R}$ is given by \eqref{CPR}.
The  dual problem of ($P$) is
\begin{align*}
(D): \qquad\quad
\begin{array}{rl}
 \vartheta_D =  \max \limits_{\lambda, s,Z}  &  b^T \lambda +C\bullet Z\\
  \mbox{s.t.}  &\sum\limits_{i=1}^m \lambda_i  a_i + s+ \operatorname{vech}(2E_n-I_n) \circ \operatorname{vech}(Z)= 0,\\
  & \lambda_i\geq 0, \ i=m_e+1,\ldots,m,\\
   & (s,(Z,1)) \in \mathcal{P} \times \mathcal{K}_p^*,
 \end{array}
\end{align*}
where $\mathcal{P}$ is given by \eqref{COP}.

By weak duality, for all feasible points $(x, Y,\gamma)$ in ($P$) and $(\lambda, s,Z)$ in ($D$), we have
\begin{align}\label{pd}
\vartheta_P\geq \vartheta_D.
\end{align}
 The theorem below shows when the strong duality holds, i.e, the equality holds for \eqref{pd}.
\begin{theorem}
\label{strongdual}
For problems ($P$) and ($D$),
\begin{itemize}
  \item[(i)] If there exist $(x^*, Y^*,\gamma^*)\in \mathcal{F}($P$)$ and $(\lambda^*, s^*,Z^*) \in \mathcal{F}($D$)$
  such that $\gamma^*-b^T \lambda^*-C\bullet Z^*=0$,
  then $(x^*, Y^*,\gamma^*)$ and $(\lambda^*,s^*,Z^*)$ are the minimizers of ($P$) and ($D$), respectively.
  \item[(ii)] If there exists $(x, Y,\gamma)$ such that $(x, Y,\gamma)\in \mathcal{F}($P$)$,
  $  a_i \bullet x  > b_i(i=m_e+1,\ldots,m)$
  and $(x, (Y,\gamma)) \in \operatorname{int}( \mathcal{R} \times \mathcal{K}_p)$,
  then we have $\vartheta_P=\vartheta_D$. Furthermore, if $\vartheta_P$ is finite, then there exists an optimal $(\lambda^*,s^*,Z^*)$ such that $\vartheta_P=
   b^T \lambda^*+C\bullet Z^*=\vartheta_D$.
  If $(x^*, Y^*,\gamma^*)$ is the minimizer of ($P$), then there exists a dual feasible point $(\overline{\lambda},\overline{s},\overline{Z})$
  such that $\gamma^*-b^T \bar\lambda-C\bullet \bar Z=0$.
  \item[(iii)] If there exists $(\lambda,s,Z)\in \mathcal{F}($D$)$ such that $\lambda_i> 0(i=m_e+1,\ldots,m)$
  and $(s,(Z,1)) \in \operatorname{int}(\mathcal{P} \times \mathcal{K}^*_p)$,
  then we have $\vartheta_P=\vartheta_D$. Furthermore, if $\vartheta_D$ is finite, then there exists an optimal $(x^*, Y^*,\gamma^*)$ such that $\vartheta_P=\gamma^*=\vartheta_D$.
  If $(\lambda^*,s^*,Z^*)$ is the minimizer of ($D$), then there exists a primal feasible  $(\overline{x},\overline{Y},\overline{\gamma})$ such that $\overline{\gamma}-b^T \lambda^*-C\bullet Z^*=0$.
\end{itemize}
\end{theorem}

Theorem \ref{strongdual} can be implied by the standard strong duality theory (cf. \cite{Ben}),
so we omit the proof here.
 For convenience, if there exists $(\lambda,s,Z)\in \mathcal{F}($D$)$ such that $\lambda_i> 0(i=m_e+1,\ldots,m)$
  and $(s,(Z,1)) \in \operatorname{int}(\mathcal{P} \times \mathcal{K}^*_p)$, we call ($D$) have a relative interior.

\section{A semidefinite algorithm}
In this section, we present a semidefinite algorithm for the CP-matrix approximation problem
and study its convergence properties.

\subsection{A semidefinite algorithm}
As shown in \eqref{h0} and \eqref{h13}, $\mathcal{R}$ and $\mathcal{P}$ have nice relaxations
$\Upsilon_k$ and $\Psi_k$, respectively. By \eqref{h11} and \eqref{h12}, the $k$-th order
relaxation of ($P$) can be defined as
\begin{align*}
(P^k): \qquad\quad
\begin{array}{rl}
 \vartheta_P^k =  \min\limits_{x, Y,\gamma, \tilde x}  &  \gamma\\
  \mbox{s.t.}   &   a_i^T x   = b_i, \ i=1,\ldots,m_e, \\
 &   a_i^T x  \geq b_i, \ i=m_e+1,\ldots,m, \\
 & x-\operatorname{vech}(Y)=\operatorname{vech}(C),\\
 & x=\tilde x|_{ E }, \\
  & (\tilde x,(Y,\gamma)) \in   \Gamma_k\times \mathcal{K}_p,
 \end{array}
\end{align*}
and the dual problem of $(P^k)$ is
\begin{align*}
(D^k): \qquad\quad
\begin{array}{rl}
 \vartheta_D^k =   \max \limits_{\lambda, s,Z}   &  b^T \lambda +C\bullet Z \\
  \mbox{s.t.}  &\sum\limits_{i=1}^m \lambda_i  a_i + s+ \operatorname{vech}(2E_n-I_n) \circ \operatorname{vech}(Z) = 0,\\
  & \lambda_i\geq 0, \ i=m_e+1,\ldots,m,\\
   & (s,(Z,1)) \in \Psi_k \times \mathcal{K}_p^*.
 \end{array}
\end{align*}
Both $(P^k)$ and $(D^k)$ are SDP problems, so they can be solved efficiently.

Clearly, $\vartheta_P^k\leq \vartheta_P$ and $\vartheta_D^k\leq \vartheta_D$ for all $k$.
Suppose $(x^{*,k}, Y^{*,k},\gamma^{*,k},\tilde x^{*,k})$ is a minimizer of ($P^k$)
and $(\lambda^{*,k},s^{*,k},Z^{*,k})$ is a maximizer of ($D^k$).
If $x^{*,k}=\tilde x^{*,k}|_{ E }\in \mathcal{R}$, then $\vartheta_P^k=\vartheta_P$
and $(x^{*,k}, Y^{*,k},\gamma^{*,k})$ is a minimizer of ($P$), i.e., the relaxation $(P^k)$ is exact for solving ($P$).
In this case, if $\vartheta_P^k=\vartheta_D^k$, then $\vartheta_D^k=\vartheta_D$ and
$(\lambda^{*,k},s^{*,k},Z^{*,k})$ is a maximizer of ($D$).
If the relaxation ($P^k$) is infeasible, then ($P$) is infeasible, i.e., \eqref{CPLS} is infeasible.

Based on the above, we propose a semidefinite algorithm for the CP-matrix approximation problem \eqref{CPLS}.

\balg[A semidefinite algorithm for the CP-matrix approximation problem]\label{Algorithm}
 \

\textbf{Step 0.} Input $C\in \mathcal{S}_n$ and $\Delta$ as (\ref{KE}). Let $k := 2$.

\textbf{Step 1.}
Solve the primal-dual pair ($P^k$)-($D^k$). If ($P^k$) is infeasible, stop
and output that ($P$) is infeasible; otherwise, compute an optimal solution
$(x^{*,k}, Y^{*,k},\gamma^{*,k},\tilde x^{*,k})$ of ($P^k$). Let $t :=1$.

\textbf{Step 2.} Let $\hat x :=\tilde x^{*,k}|_{2t}$. If the rank condition (\ref{RC})
is not satisfied, go to Step 4.

\textbf{Step 3.} Compute the finitely atomic measure $\mu$ admitted by $\hat x$:
$$\mu = \rho_1 \delta(u_1) + \cdots +\rho_r \delta(u_r),$$
where $\rho_i > 0$, $u_i \in \Delta$,  $r = \text{rank} (M_t (\hat x))$  and $\delta(u_i)$ is the Dirac
measure supported on the point $u_i\in \Delta$.

\textbf{Step 4.} If $t < k$, set $t := t+1$ and go to Step 2; otherwise, set
$k := k+1$ and go to Step 1.
 \ealg

If \eqref{CPLS} is feasible, Algorithm \ref{Algorithm} can give
 a projection matrix of a symmetric matrix onto the intersection of a set of linear constraints and the CP cone.
 A CP-decomposition of the projection matrix can also be obtained. If  \eqref{CPLS} is infeasible,
 Algorithm \ref{Algorithm} can   give a certificate for the infeasibility.

We use  Step 2  to check whether $\tilde x^{*,k}|_{2t}$ is flat or not.
Nie \cite{Nie1} showed it might be possible that $x^{*,k} \in \mathcal{R}$ while $\tilde x^{*,k}|_{2t}$ is not flat for all $t$ . In such cases,
we can apply the algorithms given in \cite{Nie, ZhouFan14a} to check whether $x^{*,k}\in \mathcal{R}$ or not.


We use Henrion and Lasserre's method in \cite{HenrionJ} to get a $r$-atomic $\Delta$-measure for $\hat x$, which can further produce the
CP-decomposition of the projection matrix.

We discuss how to solve ($P^k$) for different $p$-norm cone in section 4.3.

\subsection{Convergence properties}
We give the asymptotic convergence of Algorithm \ref{Algorithm} as follows.

\begin{theorem}
\label{Algorithmresults}
Let $ E $ and $\Delta$ be as in \eqref{AE} and \eqref{K2}, respectively.
Suppose ($P$) is feasible and ($D$)  has a relative interior point.
 Algorithm \ref{Algorithm} has the following properties:
\begin{itemize}

\item [(i)] For all $k$ sufficiently large,  ($D^k$) has a relative interior point  and ($P^k$) has a
minimizer  $(x^{*,k}, Y^{*,k},\gamma^{*,k},\tilde x^{*,k})$.

\item [(ii)]  The sequence $\{(x^{*,k}, Y^{*,k},\gamma^{*,k})\}$ is bounded, and each of its accumulation points is a
minimizer of ($P$).  The sequence $\{\gamma^{*,k}\}$ converges to the minimum of $\eqref{CPLS}$.
\end{itemize}
\end{theorem}

\begin{proof}

(i) Let $(\lambda^0, s^0,Z^0)\in \mathcal{F}($D$)$ with $\lambda^0_i> 0(i=m_e+1,\ldots,m)$,
  and $(s^0,(Z^0,1)) \in \operatorname{int}(\mathcal{P} \times \mathcal{K}_p^*)$.
Then, $s^0|_\Delta>0$ (cf. \cite[Lemma 3.1]{Nie1}).
Note that since $\Delta$ is compact,
 there exist $ \epsilon > 0$ and $ \delta > 0$ such that
$$
s|_\Delta - \epsilon > \epsilon, \quad \forall s \in B(s^0, \delta).
$$
By \cite[Theorem 6]{Nie5}, there exists $N_0 > 0$ such that
$$
s - \epsilon \in I_{2N_0} (h) + Q_{N_0}(g), \quad \forall s \in B(s^0, \delta).
$$
So ($D^k$) has a relative interior point for all $k \geq N_0$,  thus the strong duality holds for
($P^k$) and ($D^k$).
As ($P$) is feasible, the relaxation problem ($P^k$) is also feasible.
So, ($P^k$) has a minimizer $(x^{*,k}, Y^{*,k},\gamma^{*,k},\tilde x^{*,k})$ (cf. \cite[Theorem 2.4.I]{Ben}).

(ii) We first show $\{(x^{*,k}, Y^{*,k},\gamma^{*,k})\}$ is bounded.
Let $(\lambda^0, s^0,Z^0)$ and $\epsilon$ be as
in the proof of (i). The set $I_{2N_0}(h) + Q_{N_0}(g)$ is dual to $\Gamma_{N_0}$. For all $k \geq N_0$, we have $\tilde x^{*,k} \in \Gamma_{N_0}$ and
$$
\begin{array}{c}
  0 \leq \langle s^0 - \epsilon,\tilde x^{*,k}\rangle = \langle s^0,\tilde x^{*,k}\rangle - \epsilon \langle 1,\tilde x^{*,k} \rangle, \\
    \langle (s^0,Z^0,1),(x^{*,k}, Y^{*,k},\gamma^{*,k})\rangle  =  \gamma^{*,k}-b^T\lambda^0 -  C\bullet Z^0.
\end{array}
$$
Since $\gamma^{*,k}\leq \vartheta_P$ and
$\langle s^0,\tilde x^{*,k}\rangle=\langle s^0,x^{*,k}\rangle \leq \langle (s^0,Z^0,1),(x^{*,k}, Y^{*,k},\gamma^{*,k})\rangle $,
 it holds that
$$
\langle s^0,\tilde x^{*,k}\rangle \leq T_0:= \vartheta_P-b^T\lambda^0 -  C\bullet Z^0.
$$
We get
$$
\begin{array}{c}
  0 \leq \langle s^0 - \epsilon,\tilde x^{*,k}\rangle \leq T_0 -\epsilon (\tilde x^{*,k})_{\mathbf{0}}, \\
  (\tilde x^{*,k})_{\mathbf{0}} \leq T_1:= T_0/\epsilon.
\end{array}
$$
Note that $I(h) + Q(g)$ is archimedean, following the line of proof given in \cite[Theorem 4.3 (ii)]{Nie1}, we can obtain that the sequence $\{x^{*,k}\}$ is bounded.
Due to the relationships between the definitions of $x, Y$ and $\gamma$,
we know $\{(x^{*,k}, Y^{*,k},\gamma^{*,k})\}$ is bounded.

Suppose $(x^{*}, Y^{*},\gamma^{*})$ is an accumulation point of $\{(x^{*,k}, Y^{*,k},\gamma^{*,k})\}$.
Without loss of generality, we assume
$$
(x^{*,k}, Y^{*,k},\gamma^{*,k})\rightarrow (x^{*}, Y^{*},\gamma^{*}), \quad k \rightarrow \infty.
$$
Since $x^{*,k}\in \Upsilon_{k} $,
by \eqref{h13} and \eqref{h14}, we have $x^*\in \bigcap_{k=1}^{\infty} \Upsilon_k= \mathcal{R}$.
Note that $(x^{*,k}, Y^{*,k},\gamma^{*,k})\in \mathcal{F}(P^k)$,
we further  obtain $(x^{*}, Y^{*},\gamma^{*})\in \mathcal{F}(P)$.
Hence,
\begin{equation}\label{4.8}
\vartheta_P\leq \gamma^*.
\end{equation}
Since ($P^k$) is a relaxation problem of ($P$) and $(x^{*,k}, Y^{*,k},\gamma^{*,k})$ is a minimizer of ($P^k$),
 we have
\[
\vartheta_P\geq \gamma^{*,k},\ k=1,2,\ldots
\]
Taking $k\to \infty$, we get
\begin{equation}\label{4.9}
\vartheta_P\geq \lim_{k\rightarrow \infty} \gamma^{*,k}=\gamma^*,
\end{equation}
which together with \eqref{4.8} implies that
$$
\vartheta_P=\gamma^{*}.
$$
So, $(x^{*}, Y^{*},\gamma^{*})$ is a minimizer of ($P$),  and the sequence $\{\gamma^{*,k}\}$ converges to the minimum of ($P$).
\end{proof}

\begin{remark} \label{finiteconvergence}

If \eqref{CPLS} is feasible, then, under some general conditions \cite{Nie3,NieR},
which is almost necessary and sufficient, we can get a flat extension $\tilde x^{*,k}$ by solving
the hierarchy of ($P^k$), within finitely many steps \cite[Section 4]{Nie1}.

\end{remark}

\subsection{Subproblem solving}
   We discuss how to solve the subproblem ($P^k$) in Algorithm \ref{Algorithm} for different $p$-norm cone $\mathcal{K}_p$ ($p=1,2,\infty, F$).

\smallskip
\textbf{1. $1$-norm or $\infty$-norm cone.}
The $1$-norm and $\infty$-norm are the same for symmetric matrices.
Let $Y=Y^+ - Y^-$, where $Y^+, Y^- \geq 0$ and $Y^+, Y^-\in \mathcal{S}_n$.
Then ($P^k$) can be transformed to the following problem:
\begin{align}\label{pkl1}
\begin{array}{cl}
 \min\limits_{\mathbf{x},Y^+,Y^-, \gamma, \tilde x}  & \gamma\\
  \mbox{s.t.}  &  a_i^T x   = b_i, \ i=1,\ldots,m_e, \\
 &   a_i^T x  \geq b_i, \ i=m_e+1,\ldots,m, \\
 & x-\operatorname{vech}(Y^+ - Y^-)=\operatorname{vech}(C),\\
 & \tilde E_{j}\bullet (Y^+ + Y^-) \leq \gamma, \ j=1,\ldots, n, \\
 & Y^+, Y^- \geq 0, Y^+, Y^-\in \mathcal{S}_n,\\
 & x=\tilde x|_{ E }, \tilde x \in   \Gamma_k,
 \end{array}
\end{align}
where $\tilde E_{j}$ is the matrix whose $j$-th column is of all ones and other entries are zeros.

\eqref{pkl1}  is a linear optimization problem with linear matrix inequalities.
It can be solved by the softwares GloptiPoly 3 \cite{HenrionJJ} and SeDuMi \cite{Sturm}.

\smallskip
\textbf{2. $2$-norm cone.}
Note that $(Y,\gamma)\in \mathcal{K}_2$ if and only if
$\left(
                   \begin{array}{cc}
                     \gamma I_n & Y \\
                     Y^T & \gamma I_n \\
                   \end{array}
                 \right)\succeq 0.$
Since $Y=X-C$, we can transform ($P^k$) to the following problem:
\begin{align}\label{pkl2}
\begin{array}{cl}
 \displaystyle  \min\limits_{x, \gamma, \tilde x}  & \gamma\\
  \mbox{s.t.} &   a_i^T x   = b_i, \ i=1,\ldots,m_e, \\
 &   a_i^T x  \geq b_i, \ i=m_e+1,\ldots,m, \\
 & \left(
                   \begin{array}{cc}
                     \gamma I_n & \operatorname{vech}^{-1}(x)-C \\
                     (\operatorname{vech}^{-1}(x)-C)^T & \gamma I_n \\
                   \end{array}
                 \right)\succeq 0, \\
  & x=\tilde x|_{ E },  \tilde x \in   \Gamma_k,
 \end{array}
\end{align}
where $\operatorname{vech}^{-1}(\cdot)$ denotes the inverse of the linear operator $\operatorname{vech}(\cdot)$.

\eqref{pkl2} can also be solved by the softwares GloptiPoly 3 \cite{HenrionJJ} and SeDuMi \cite{Sturm}.

\smallskip
\textbf{3. $F$-norm cone.}
Let
\begin{align*}
y= \operatorname{vech}(\sqrt{2}{E}_n+(1-\sqrt{2})I_n) \circ \operatorname{vech}(Y).
\end{align*}
Then $(Y,\gamma)\in \mathcal{K}_F$ if and only if $(y,\gamma)\in \mathcal{L}_{\bar n+1},$
where
$$
\mathcal{L}_{\bar n+1}= \{(y,\gamma)\in \mathbb{R}^{\bar n+1} :  \|y\|_2 \leq\gamma \}
$$
is the second-order cone (or Lorentz cone).
Since $Y=X-C$, ($P^k$) can be transformed to the following problem:
\begin{align}\label{pkF}
\begin{array}{rl}
 \displaystyle  \min\limits_{x, y, \gamma, \tilde x}  & \gamma\\
  \mbox{s.t.} &   a_i^T x   = b_i, \ i=1,\ldots,m_e, \\
 &   a_i^T x  \geq b_i, \ i=m_e+1,\ldots,m, \\
 & y= \operatorname{vech}(\sqrt{2}{E}_n+(1-\sqrt{2})I_n) \circ (x-\operatorname{vech}(C)), \\
  & x=\tilde x|_{ E }, \\
  & (\tilde x, (y,\gamma)) \in   \Gamma_k \times \mathcal{L}_{\bar n+1}.
 \end{array}
\end{align}

\eqref{pkF} is a linear optimization problem with the second-order cone and linear matrix inequalities.
It can be solved by the softwares GloptiPoly 3 \cite{HenrionJJ} and SeDuMi \cite{Sturm}.

\section{Numerical experiments}
In this section, we present some numerical experiments for computing the projection of a matrix
onto the intersection of a set of linear constraints and the CP cone  by using Algorithm \ref{Algorithm}.
A CP-decomposition of the projection matrix is also given if the problem is feasible.
The experiments are implemented on a laptop  with an Intel Core i5-2520M CPU
and 4GB of RAM, using Matlab R2012b. We only display 4 digits for each number.

\subsection{CP-approximation in $1$-norm or $\infty$-norm}

\bex\label{ex2l1} \upshape
Consider the symmetric matrix $C$ given as:
\begin{equation}\label{Exa2l1}
C=\left(\begin{array}{cccc}
     2   &  1   &  1   &  1\\
     1    & 2   &  2   &  1\\
     1   &  2&     6  &   5\\
     1   &  1 &    5  &   6\\
\end{array}\right).
\end{equation}
It can be checked that $C$ is double nonnegative.
Since a symmetric double nonnegative matrix with the order less than or equal to 4 is CP (cf. \cite{BermanN}), we have $C\in \mathcal{CP}_4$.

\smallskip
\textbf{Case 1.} Consider \eqref{CPLS} without linear constraints,
i.e., we compute the projection of $C$ onto $\mathcal{CP}_4$ in $1$-norm.

Algorithm \ref{Algorithm} terminates at $k=3$, with
$\gamma^{*,k}=0.0000$ and $x^{*,k}\in \mathcal{R}$.
So, $X^*=C$. This verifies that $C$ is CP.
The CP-decomposition of $C$ is $C =\sum_{i=1}^{5} \rho_i u_i u_i^T$,
where the points and their weights are:
\begin{eqnarray*}\label{Exa2del1}
&\rho_1=3.0297,  & u_1=(0.0000,
    0.6287,
    0.6491,
    0.4284)^T, \nonumber \\
&\rho_2=7.6746,  & u_2=(0.0000,
    0.0000,
    0.6347,
    0.7728)^T, \nonumber \\
&\rho_3=2.6969,  & u_3=(0.4767,
    0.3641,
    0.7779,
    0.1875)^T, \\
&\rho_4=1.0808,  & u_4=(0.7669,
    0.6418,
    0.0000,
    0.0000)^T, \nonumber\\
&\rho_5=1.5179,  & u_5=(0.7036,
    0.0000,
    0.0000,
    0.7106)^T.\nonumber
\end{eqnarray*}

\smallskip
\textbf{Case 2.} Consider \eqref{CPLS} with the CP cone and the linear constraints $A_i \bullet X= b_i (i=1,2)$,
where
\begin{align*}
& A_1=I_4,
    \quad A_2=\left(
                  \begin{array}{cccc}
                    0 & 1 & 0 & 1 \\
                    1 & 0 & 1 & 0 \\
                    0 & 1 & 0 & 1 \\
                    1 & 0 & 1 & 0 \\
                  \end{array}
                \right),
\\
& b_1=10, \quad b_2=12.
\end{align*}

Algorithm \ref{Algorithm} terminates at $k=3$  with $\gamma^{*,k}=3.0209$ and $x^{*,k}\in \mathcal{R}$.
The optimal solution is
\begin{equation*}\label{Exa1l1case2}
X^*=\left(\begin{array}{cccc}
    0.2709  &  0.1572  &  0.9582 &   0.5928\\
    0.1572  &  0.6302  &  1.1918 &   1.0000\\
    0.9582  &  1.1918  &  4.7709  &  4.0582\\
    0.5928  &  1.0000   & 4.0582  &  4.3280
\end{array}\right).
\end{equation*}
The CP-decomposition of $X^*$ is $X^* =\sum_{i=1}^{3} \rho_i u_i u_i^T$,
where the points and their weights are:
\begin{eqnarray*}\label{Exa1l1case2de}
&\rho_1=0.8735,  & u_1=(0.2014,
    0.6294,
    0.7506,
    0.0000)^T, \nonumber \\
&\rho_2=3.1791,  & u_2=(0.2677,
    0.0000,
    0.8209,
    0.5044)^T, \nonumber \\
&\rho_3=5.9473,  & u_3=(0.0357,
    0.2186,
    0.5993,
    0.7693)^T. \nonumber
\end{eqnarray*}

\smallskip
\textbf{Case 3.}  Consider \eqref{CPLS} with the CP cone and the linear constraints $A_i \bullet X= b_i (i=1,2)$, where
\begin{align*}
&  A_1=\left(
      \begin{array}{cccc}
        1 & -1 & 1 & -1  \\
        -1 & 2 & -2 & 2 \\
        1 & -2 & 3 &-3\\
        -1 & 2 & -3 & 4
      \end{array}
    \right), \quad A_2=-I_4
    \\
 & b_1=5, \quad b_2=-19.
\end{align*}

Algorithm \ref{Algorithm} terminates at $k=2$ as ($P^k$) is infeasible.
So, \eqref{CPLS} is infeasible.

\smallskip
\textbf{Case 4.} Consider \eqref{CPLS} with the CP cone and the linear constraints
 $A_1 \bullet X= b_1$ and $A_2 \bullet X\geq b_2$,
 where $A_i, b_i (i=1,2)$ are the same as in Case 3.

 Algorithm \ref{Algorithm} terminates at $k=3$ with $\gamma^{*,k}=1.6916$.
 The optimal solution is
\begin{equation*}\label{Exa1l1case2}
X^*=\left(\begin{array}{cccc}
    0.8232  &  0.8946  &  1.4094  &  1.0000\\
    0.8946   & 0.9745   & 1.4394  &  1.0000\\
    1.4094  &  1.4394  &  5.9043 &   4.9991\\
    1.0000  &  1.0000  &  4.9991  &  4.3094\\
\end{array}\right).
\end{equation*}
The CP-decomposition of $X^*$ is $X^* =\sum_{i=1}^{2} \rho_i u_i u_i^T$,
where the points and their weights are:
\begin{eqnarray*}\label{Exa2l1c4}
&\rho_1=2.0184,  & u_1=(0.5817,
    0.6464,
    0.4532,
    0.1957)^T, \nonumber \\
&\rho_2=9.9929,  & u_2=(0.1184,
    0.1145,
    0.7412,
    0.6508)^T. \nonumber
\end{eqnarray*}
 \eex

\subsection{CP-approximation in 2-norm}

\bex\label{SponselExample22norm} \upshape
 Consider the symmetric matrix $C$ given as (cf. \cite{SponselD}):
\begin{equation}\label{SponselExa}
C=\left(\begin{array}{ccccc}
  2&1&1&1&2\\
  1&2&2&1&1\\
  1&2&6&5&1\\
  1&1&5&6&2\\
  2&1&1&2&3
\end{array}\right).
\end{equation}
It is shown in \cite{SponselD} that $C\in \mathcal{CP}_5$ and the CP-rank of $C$ is 5.

\smallskip
\textbf{Case 1.}  Consider \eqref{CPLS} with the CP cone and the linear constraints $A_i \bullet X= b_i (i=1,2, 3)$, where
\begin{align*}
&A_1=I_5, \quad A_2=\left(
      \begin{array}{ccccc}
        1 & -1 & 1 & -1 & 1 \\
        -1 & 2 & -2 & 2 & -2 \\
        1 & -2 & 3 & -3 & 3 \\
        -1 & 2 & -3 & 4 & -4 \\
        1 & -2 & 3 & -4 & 5
      \end{array}
    \right),
    \quad A_3=\left(
                  \begin{array}{ccccc}
                    0 & 1 & 0 & 1 & 0 \\
                    1 & 0 & 1 & 0 & 1 \\
                    0 & 1 & 0 & 1 & 0 \\
                    1 & 0 & 1 & 0 & 1 \\
                    0 & 1 & 0 & 1 & 0
                  \end{array}
                \right),
\\
& b_1=19, \quad b_2=17, \quad b_3=24.
\end{align*}

Algorithm \ref{Algorithm} terminates at $k=3$, with
$\gamma^{*,k}=0.0000$ and $x^{*,k}\in \mathcal{R}$.
So, $X^*=C$. This implies that $C$ is not only CP but also satisfies the linear constraints.
The CP-decomposition of $C$ is $C =\sum_{i=1}^{5} \rho_i u_i u_i^T$,
where the points and their weights are:
\begin{eqnarray*}\label{Exa4deF}
&\rho_1=5.5421,  & u_1=(0.0862, 0.0000, 0.3963, 0.7926, 0.4553)^T, \nonumber \\
&\rho_2=3.8751,  & u_2=(0.6826, 0.2828, 0.0000, 0.0000, 0.6738)^T, \nonumber \\
&\rho_3=7.2866,  & u_3=(0.1450, 0.2334, 0.7608, 0.5879, 0.0000)^T, \\
&\rho_4=0.8380,  & u_4=(0.0000, 0.9438, 0.0000, 0.0000, 0.3306)^T, \nonumber\\
&\rho_5=1.4582,  & u_5=(0.0058, 0.6122, 0.7907, 0.0000, 0.0000)^T.\nonumber
\end{eqnarray*}

We obtained a minimal CP-decomposition of $C$.
It is different from the minimal CP-decomposition given in \cite{SponselD}.

\smallskip
\textbf{Case 2.} Consider \eqref{CPLS} with the CP cone and the linear constraints $A_i \bullet X= b_i (i=1,2, 3)$, where  $A_1, A_2, A_3, b_1, b_3$ are the same as in Case 1, and
$$
b_2=50.
$$

Algorithm \ref{Algorithm} terminates at $k=3$  with $\gamma^{*,k}=2.8436$.
The optimal solution is
\begin{equation*}\label{Exa4proFlinear2norm}
X^*=\left(\begin{array}{ccccc}
    1.6135 &   1.3913 &   2.3928  &  1.6291  &  2.3642\\
    1.3913  &  2.4173 &   2.1301  &  3.1909  &  1.6217\\
    2.3928  &  2.1301 &   5.0301  &  4.1451  &  2.8557\\
    1.6291  &  3.1909 &   4.1451  &  6.0656  &  1.0827\\
    2.3642  &  1.6217 &   2.8557  &  1.0827  &  3.8735\\
\end{array}\right).
\end{equation*}
The CP-decomposition of $X^*$ is $X^* =\sum_{i=1}^{3} \rho_i u_i u_i^T$,
where the points and their weights are:
\begin{eqnarray*}\label{Exa4deFlinear2norm}
&\rho_1=6.4943,  & u_1=(0.4251,
    0.1801,
    0.4791,
    0.0000,
    0.7465)^T, \nonumber \\
&\rho_2=3.9517,  & u_2=(0.1403,
    0.0000,
    0.7829,
    0.6061,
    0.0000)^T, \nonumber \\
&\rho_3=8.5539,  & u_3=(0.2058,
    0.5079,
    0.3613,
    0.7344,
    0.1723)^T. \nonumber
\end{eqnarray*}

\smallskip
 \textbf{Case 3.} Consider \eqref{CPLS} with the CP cone and the linear constraints $A_i \bullet X= b_i (i=1,2, 3)$, where $A_1, A_2, A_3, b_1, b_3$ are  the same as in Case 1, but
 $$
 b_2=-50.
 $$

Algorithm \ref{Algorithm} terminates at $k=2$ as ($P^k$) is infeasible.
So, \eqref{CPLS} is infeasible.

\smallskip
\textbf{Case 4.} Consider \eqref{CPLS} with the CP cone and the linear constraints $A_i \bullet X= b_i (i=1,2)$ and $A_3 \bullet X \geq b_3$,
 where $A_i (i=1,2,3)$ are the same as in Case 1, and
 $$
 b_1=10,\quad b_2=12,\quad b_3=-2.
 $$

  Algorithm \ref{Algorithm} terminates at $k=3$, with
$\gamma^{*,k}=3.3763$ and $x^{*,k}\in \mathcal{R}$.
The optimal solution is:
\begin{equation}\label{Exa42normcase4pro}
X^*=\left(\begin{array}{ccccc}
    0.4943  &  0.3541  &  1.2119 &   0.9809  &  0.7703\\
    0.3541  &  0.4565 &   1.3018 &   1.3352  &  0.4720\\
    1.2119  &  1.3018 &   3.8986  &  3.7575  &  1.7180\\
    0.9809  &  1.3352  &  3.7575  &  3.9192  &  1.2798\\
    0.7703  &  0.4720 &   1.7180  &  1.2798  &  1.2316\\
\end{array}\right).
\end{equation}
The CP-decomposition of $X^*$ is $X^* =\sum_{i=1}^{2} \rho_i u_i u_i^T$,
where the points and their weights are:
\begin{eqnarray*}\label{Ex2normcase4F}
&\rho_1=1.3602,  & u_1=(0.4277,
    0.0342,
    0.4667,
    0.0000,
    0.7734)^T, \nonumber \\
&\rho_2=8.6396,  & u_2=(0.1686,
    0.2295,
    0.6457,
    0.6735,
    0.2199)^T. \nonumber
\end{eqnarray*}
 \eex

\subsection{CP-approximation in $F$-norm}

\bex\label{SponselExample} \upshape
Consider the symmetric matrix $C$ given in Example \ref{SponselExample22norm}.

\smallskip
 \textbf{Case 1.}   Consider \eqref{CPLS} with the CP cone and the linear constraints $A_i \bullet X= b_i (i=1,2, 3)$,
 where $A_i, b_i (i=1,2,3)$ are  the same as in Case 1 of Example \ref{SponselExample22norm}.

 Algorithm \ref{Algorithm} terminates at $k=3$, with
$\gamma^{*,k}=0.0000$ and $x^{*,k}\in \mathcal{R}$. So,  $C\in \mathcal{CP}_5$.
We get the same CP-decomposition of $C$ as that in Case 1 of Example \ref{SponselExample22norm}.


\smallskip
 \textbf{Case 2.} Consider \eqref{CPLS} with the CP cone and the linear constraints $A_i \bullet X= b_i (i=1,2, 3)$, where $A_i, b_i (i=1,2,3)$ are  the same as in Case  2 of Example \ref{SponselExample22norm}.

Algorithm \ref{Algorithm} terminates at $k=3$  with $\gamma^{*,k}=4.7642$. The optimal solution is
\begin{equation*}\label{Exa4proFlinear}
X^*=\left(\begin{array}{ccccc}
  1.4456&1.2748&2.0122&1.7636&2.1333\\
  1.2748&1.3345&2.1511&2.4889&1.4697\\
  2.0122&2.1511&5.5864&4.4398&3.0769\\
  1.7636&2.4889&4.4398&6.3423&0.9011\\
  2.1333&1.4697&3.0769&0.9011&4.2912
\end{array}\right).
\end{equation*}
The CP-decomposition of $X^*$  is $X^* =\sum_{i=1}^{4} \rho_i u_i u_i^T$,
where the points and their weights are:
\begin{eqnarray*}\label{Exa4deFlinear}
&\rho_1=2.5727,  & u_1=(0.5085, 0.3250, 0.0000, 0.0000, 0.7973)^T, \nonumber \\
&\rho_2=7.8537,  & u_2=(0.1698, 0.2829, 0.4990, 0.8013, 0.0000)^T, \nonumber \\
&\rho_3=4.5924,  & u_3=(0.1992, 0.1020, 0.7125, 0.0000, 0.6650)^T, \\
&\rho_4=3.9812,  & u_4=(0.3055, 0.3115, 0.5712, 0.5713, 0.3962)^T. \nonumber
\end{eqnarray*}

\smallskip
 \textbf{Case 3.} Consider \eqref{CPLS} with the CP cone and the linear constraints $A_i \bullet X= b_i (i=1,2, 3)$, where $A_i, b_i (i=1,2,3)$ are  the same as in Case 3 of Example \ref{SponselExample22norm}.

Algorithm \ref{Algorithm} terminates at $k=2$ as ($P^k$) is infeasible.
So, \eqref{CPLS} is infeasible.

\smallskip
 \textbf{Case 4.}   Consider \eqref{CPLS} with the CP cone and the linear constraints $A_i \bullet X= b_i (i=1,2)$ and $A_3 \bullet X \geq b_3$,
 where $A_i (i=1,2,3)$ are the same as in Case 4 of Example \ref{SponselExample22norm}.

  Algorithm \ref{Algorithm} terminates at $k=3$, with
$\gamma^{*,k}=5.1904$ and $x^{*,k}\in \mathcal{R}$.
The optimal solution is:
\begin{equation}\label{Exa4case4pro}
X^*=\left(\begin{array}{ccccc}
    0.6441  &  0.3961  &  1.1295   & 0.8842  &  0.9479\\
    0.3961  &  0.4305 &   1.2678  &  1.2371  &  0.5508\\
    1.1295  &  1.2678  &  3.7388  &  3.6771  &  1.5635\\
    0.8842  &  1.2371  &  3.6771  &  3.7860  &  1.1818\\
    0.9479  &  0.5508  &  1.5635  &  1.1818   & 1.4007
\end{array}\right).
\end{equation}
The CP-decomposition of $X^*$  is $X^* =\sum_{i=1}^{2} \rho_i u_i u_i^T$,
where the points and their weights are:
\begin{eqnarray*}\label{Ex4case4F}
&\rho_1=1.6630,  & u_1=(0.5129,
    0.1257,
    0.3173,
    0.0000,
    0.7877)^T, \nonumber \\
&\rho_2=8.3370,  & u_2=(0.1574,
    0.2202,
    0.6545,
    0.6739,
    0.2104)^T. \nonumber
\end{eqnarray*}
 \eex

\bex\label{randomExample} \upshape
  Consider the symmetric matrix
\begin{equation*}\label{randex}
C=\left(\begin{array}{cccccc}
     4   &  5  &   4  &   6    & 4  &   2\\
     5  &   1   &  4  &   7    & 4  &   6\\
     4   &  4  &   4  &   2  &   5  &   4\\
     6  &   7  &   2  &   0   &  3   &  7\\
     4   &  4 &    5 &    3  &   1   &  6\\
     2&     6   &  4    & 7    & 6    & 4\\
\end{array}\right),
\end{equation*}
 which is generated randomly in Matlab.

\smallskip
  \textbf{Case 1.} Consider \eqref{CPLS} without linear constraints.
Algorithm \ref{Algorithm} terminates at $k=3$, with $\gamma^{*,k}=9.7852$.
So, $C$ is not CP.
The projection matrix of $C$ onto $\mathcal{CP}_5$ is
\begin{equation*}\label{rand1proF}
X^*=\left(\begin{array}{cccccc}
    5.3184  &  4.4216  &  3.6259 &   4.2906  &  3.4447 &   3.6739\\
    4.4216    &4.8771  &  3.6321 &   4.7517  &  3.9731 &   5.1854\\
    3.6259   & 3.6321 &   4.7970 &   3.0763  &  3.9514  &  3.9240\\
    4.2906  &  4.7517  &  3.0763 &   4.7343   & 3.6374   & 4.9919\\
    3.4447   & 3.9731 &   3.9514  &  3.6374   & 3.7868  &  4.5142\\
    3.6739   & 5.1854  &  3.9240  &  4.9919  &  4.5142 &   6.3710\\
\end{array}\right).
\end{equation*}
The CP-decomposition of $X^*$ is $X^* =\sum_{i=1}^{4} \rho_i u_i u_i^T$,
where the points and their weights are:
\begin{eqnarray*}\label{Exa4deF}
&\rho_1=7.0443,  & u_1=(0.3599,
    0.4862,
    0.0000,
    0.5526,
    0.2300,
    0.5252)^T, \nonumber \\
&\rho_2=4.6526,  & u_2=(0.0000,
    0.2673,
    0.6042,
    0.1571,
    0.4868,
    0.5494)^T, \nonumber \\
&\rho_3=13.2067,  & u_3=(0.3707,
    0.4278,
    0.3766,
    0.4028,
    0.3828,
    0.4785)^T, \nonumber \\
&\rho_4=4.9810,  & u_4=(0.7213,
    0.3047,
    0.4961,
    0.2554,
    0.2749,
    0.0000)^T. \nonumber
\end{eqnarray*}

\smallskip
 \textbf{Case 2.}  Consider \eqref{CPLS} with the CP cone and the linear constraints $A_i \bullet X= b_i (i=1,2)$, where $A_i, b_i (i=1,2)$ are generated randomly:
\begin{align*}
& A_1=\left(
      \begin{array}{cccccc}
   -12  &   0  &   7 &   -5   &  4  &  -2\\
     0  &   3  &   1 &   -2  &  -6  & -13\\
     7   &  1  &   4  &   1   & -9   &  6\\
    -5  &  -2   &  1  &   7  &  -9   & 10\\
     4   & -6  &  -9  &  -9   &-19  &   1\\
    -2  & -13   &  6 &   10   &  1   & 13
      \end{array}
    \right),
\\
& A_2=\left(
                  \begin{array}{cccccc}
    -4 &    3  &  11   & 11   &  2  &  -5\\
     3    & 6   &  3   & -3  &   5  &  -9\\
    11   &  3   &  5   &  0  &  -3  &  -9\\
    11  &  -3   &  0  &  14  &  -4  & -16\\
     2  &   5  &  -3   & -4  &   7  & -14\\
    -5   & -9  &  -9  & -16  & -14 &    3\\
                  \end{array}
                \right),
\\ &
b_1=-17,\quad b_2=6.
\end{align*}

Algorithm \ref{Algorithm} terminates at $k=3$, with $y_1^{*,k}=11.4970$.
The optimal solution is
\begin{equation*}\label{rand1proFlinear}
X^*=\left(\begin{array}{cccccc}
    5.5277 &   4.9260  &  5.0372   & 4.5381  &  3.1904  &  2.7187\\
    4.9260 &   4.7114  &  4.2671 &   4.8028  &  3.0364  &  3.5882\\
    5.0372   & 4.2671  &  6.0026 &   3.3516  &  3.4975  &  3.0055\\
    4.5381  &  4.8028  &  3.3516 &   5.5691  &  2.9252  &  4.7051\\
    3.1904   & 3.0364  &  3.4975 &   2.9252  &  2.3729  &  3.0340\\
    2.7187   & 3.5882  &  3.0055 &   4.7051 &   3.0340  &  6.9676\\
\end{array}\right).
\end{equation*}
The CP-decomposition of $X^*$ is $X^* =\sum_{i=1}^{3} \rho_i u_i u_i^T$,
where the points and their weights are:
\begin{eqnarray*}\label{Exa4deFlinear}
&\rho_1=5.0609,  & u_1=(0.5069,
    0.5243,
    0.0000,
    0.6726,
    0.0996,
    0.0766)^T, \nonumber \\
&\rho_2=7.4131,  & u_2=(0.0000,
    0.1967,
    0.0000,
    0.4359,
    0.1960,
    0.8561)^T, \nonumber \\
&\rho_3=18.6772,  & u_3=(0.4757,
    0.4030,
    0.5669,
    0.3165,
    0.3303,
    0.2839)^T. \nonumber
\end{eqnarray*}

\smallskip
   \textbf{Case 3.}  Consider \eqref{CPLS} with the CP cone and the linear constraints $A_i \bullet X= b_i (i=1,2)$, where $A_i, b_i (i=1,2)$ are generated randomly:
\begin{align*}
& A_1=\left(
      \begin{array}{cccccc}
     8  &   -2  &  5&    6   &  5   &  -4\\
    -2 &   10   &  8 &   12  &  17  &   4\\
     5  &   8 &    7   &  6  &  -2  &   -3\\
     6  &  12  &   6   &  4  &  12  &   7\\
     5  &  17  &  -2   & 12   & 10  &  -8\\
     -4   &  4&     -3  &   7 &   -8  &   9
      \end{array}
    \right),
 \\
 & A_2=\left(
                  \begin{array}{cccccc}
    -2  & -16  & -12 &    4  &   1   & -5\\
   -16  &   3 &    8  &  -3 &  -10   &  0\\
   -12  &   8  & -13  &  -1 &   11  &   3\\
     4   & -3  &  -1  &  -3 &    5   &  9\\
     1  & -10  &  11  &   5  &  10   &  3\\
    -5   &  0  &   3  &   9 &    3  & -15
                  \end{array}
                \right),
    \\ & b_1=-6,\quad b_2=4.
\end{align*}
Algorithm \ref{Algorithm} terminates at $k=2$ as ($P^k$) is infeasible.
So, \eqref{CPLS} is infeasible.

\smallskip
  \textbf{Case 4.}  Consider \eqref{CPLS} with the CP cone and the linear constraints $A_1 \bullet X= b_1, A_2 \bullet X\geq b_2$, where $A_i, b_i (i=1,2)$ are generated randomly:
\begin{align*}
& A_1=\left(
      \begin{array}{cccccc}
     5  &   7  &  -4  &  -9 &    4  &   9\\
     7   & -2  &   6 &   -4    & 7  &  -6\\
    -4  &   6  & -17 &   -9 &   -1  &   6\\
    -9  &  -4  &  -9 &    5 &  -13  &   6\\
     4  &   7  &  -1 &  -13&    -3 &    1\\
     9   & -6  &   6   &  6 &    1  &  -6
      \end{array}
    \right),
 \\
 & A_2=\left(
                  \begin{array}{cccccc}
     2  &  -4  &   6   &  4   &  7 &    1\\
    -4   & -2  &  11   &  2  &   6  &   7\\
     6  &  11  &  12  &  -9  & -2    & 7\\
     4  &   2   & -9  &  -3  &   0  &  10\\
     7   &  6   & -2   &  0  &   4  & -11\\
     1 &    7     &7    &10  & -11  &  11
                  \end{array}
                \right),
    \\ & b_1=7,\quad b_2=-10.
\end{align*}
Algorithm \ref{Algorithm} terminates at $k=3$, with $\gamma^{*,k}=10.4410$.
The optimal solution is
\begin{equation*}\label{rand1proFlinear}
X^*=\left(\begin{array}{cccccc}
    5.5853   & 4.9676  &  3.4690  &  3.9792 &   3.9314 &   4.6082\\
    4.9676 &   5.0242  &  3.6848   & 4.2784 &   4.1044  &  5.3069\\
    3.4690  &  3.6848  &  3.6883  &  2.3311  &  3.7530  &  3.9202\\
    3.9792  &  4.2784  &  2.3311   & 4.5442 &   2.8846   & 4.8841\\
    3.9314  &  4.1044  &  3.7530   & 2.8846 &   3.9128  &  4.3521\\
    4.6082  &  5.3069  &  3.9202  &  4.8841 &   4.3521  &  6.2313\\
\end{array}\right).
\end{equation*}
The CP-decomposition of $X^*$ is $X^* =\sum_{i=1}^{3} \rho_i u_i u_i^T$,
where the points and their weights are:
\begin{eqnarray*}\label{Exa4deFlinear}
&\rho_1=3.4136,  & u_1=(0.7633,
    0.3769,
    0.3647,
    0.0000,
    0.3772,
    0.0000)^T, \nonumber \\
&\rho_2=7.4149,  & u_2=(0.4695,
    0.4255,
    0.0000,
    0.6215,
    0.1124,
    0.4467)^T, \nonumber \\
&\rho_3=18.1576,  & u_3=(0.3287,
    0.4196,
    0.4221,
    0.3042,
    0.4285,
    0.5115)^T. \nonumber
\end{eqnarray*}
\eex

\bex\label{randomExample} \upshape
Consider the computing time of projecting a random symmetric matrix onto the CP cone.
For each $n=2,3,\ldots,10$, we generate 50 random symmetric $n\times n$ matrices.

Table 1 shows the average time (seconds) consumed by Algorithm \ref{Algorithm}
to compute the projection matrix onto the CP cone.

\begin{table}[htbp]\label{table}
\begin{center}
\begin{tabular}{cccccccccc} \hline
n & 2 & 3 & 4 & 5 & 6 & 7 & 8 & 9 & 10 \\ \hline
Time & 0.36 & 0.53 & 0.80 & 1.56 & 4.76 & 21.37 & 101.29 & 428.21 & 1732.21\\ \hline
\end{tabular}
\end{center}
\caption{The average time for computing the CP projection matrix.}
\end{table}
\eex

\section{Conclusions}
We study the CP-matrix approximation problem of projecting a symmetric matrix
onto the intersection of a set of linear constraints and the CP cone.
It includes the feasibility problem and the CP projection problem as special cases.
We formulate the problem as the linear optimization with the cone of moments and the $p$-norm cone ($p=1,2,\infty$, or $F$).
A semidefinite algorithm (i.e., Algorithm \ref{Algorithm}) is presented for it.
Its convergence is also studied.
If the problem is infeasible, we can get a certificate for it.
If the problem is feasible, we can get a projection matrix;
moreover, a CP-decomposition of the projection matrix can also be obtained.
Numerical results show that Algorithm \ref{Algorithm} is efficient in solving the CP-matrix approximation problem.

\end{document}